\documentclass[12pt,reqno]{amsart}

\usepackage[margin=1in]{geometry}
\usepackage{amsmath,amssymb,amsthm,mathtools}
\usepackage{stmaryrd}
\usepackage{enumitem}
\usepackage[numbers,square,sort&compress]{natbib}
\usepackage{booktabs}
\usepackage{float}
\usepackage{tikz}
\usepackage{tikz-cd}
\usetikzlibrary{calc,positioning,arrows.meta}
\usepackage{xcolor}
\usepackage{microtype}
\usepackage[hidelinks]{hyperref}
\newcommand{\doi}[1]{\url{https://doi.org/#1}}
\hypersetup{pdftitle={Graphoid Automata, Frobenius PROPs, and Semiring Evaluation},pdfauthor={Antonios Kalampakas},pdfkeywords={graphoid automata; Frobenius monoid; cospan; corelation; commutative semiring; matrix evaluation}}
\allowdisplaybreaks[2]
\setlist[enumerate]{label=\textup{(\arabic*)},leftmargin=2.2em}

\theoremstyle{plain}
\newtheorem{theorem}{Theorem}[section]
\newtheorem{proposition}[theorem]{Proposition}
\newtheorem{lemma}[theorem]{Lemma}
\newtheorem{corollary}[theorem]{Corollary}

\theoremstyle{definition}
\newtheorem{definition}[theorem]{Definition}
\newtheorem{example}[theorem]{Example}
\newtheorem{notation}[theorem]{Notation}

\theoremstyle{remark}
\newtheorem{remark}[theorem]{Remark}

\newcommand{\Rel}{\mathrm{Rel}}
\newcommand{\BMat}{\mathrm{BMat}}
\newcommand{\KMat}{\mathrm{KMat}}

\newcommand{\GR}{\mathrm{GR}}
\newcommand{\GRcore}{\GR_{\mathrm{core}}}
\newcommand{\Cospan}{\mathrm{Cospan}}
\newcommand{\Corel}{\mathrm{Corel}}
\newcommand{\FinSet}{\mathrm{FinSet}}

\newcommand{\PROP}{PROP}
\newcommand{\bx}{\boxtimes}
\newcommand{\comp}{\circ}
\newcommand{\id}{\mathrm{id}}
\newcommand{\dmul}{d_{21}}
\newcommand{\dunit}{d_{01}}
\newcommand{\dcomul}{d_{12}}
\newcommand{\dcounit}{d_{10}}
\newcommand{\swp}{s}
\newcommand{\sem}[1]{\llbracket #1 \rrbracket}

\begin{document}

\title[Graphoid Automata and Semiring Evaluation]{Graphoid Automata, Frobenius PROPs, and Semiring Evaluation}

\author{Antonios Kalampakas}
\address{College of Engineering, American University of the Middle East, Egaila 54200, Kuwait}
\email{Antonios.Kalampakas@aum.edu.kw}
\thanks{Corresponding author: Antonios Kalampakas, \href{mailto:Antonios.Kalampakas@aum.edu.kw}{\texttt{Antonios.Kalampakas@aum.edu.kw}}.}


\begin{abstract}
Graphoid automata evaluate graphs by composing relations between words of states. We examine the equations that make this evaluation independent of the chosen graph expression, and what changes when relations are replaced by matrices over a commutative semiring. We prove that the classical graphoid equations are equivalent to the special commutative Frobenius equations together with the required symmetry laws. The known cospan and hypergraph presentations then describe equality of structural and labelled expressions, with the bone law deleting precisely the isolated internal vertices. We use established partition representations to compare this equality with equality of evaluated matrices. For any nonempty finite state set $Q$ and commutative semiring $K$, we give a uniform criterion for two structural expressions to have the same matrix. The criterion determines which isolated-vertex counts are distinguished and when a vanishing scalar makes matrices from different boundary partitions equal. For the canonical structure, the bone law holds precisely when $|Q|\cdot1_K=1_K$. Preserving all relational compositions is a separate requirement, equivalent to additive idempotency for $|Q|\geq2$. We then examine the characteristic matrices of the known relational structures on $Q$ given by disjoint unions of Abelian groups. Their Frobenius equations hold over every commutative semiring, whereas specialness depends on the order of each group and the bone law depends on the number of groups. These conditions specify which equations remain valid when weights are assigned to a graphoid automaton.
\end{abstract}

\maketitle

\section{Introduction}\label{sec:intro}

In this paper we consider graphs with an ordered list of input vertices and an ordered list of output vertices. These lists specify the boundary positions used to join graphs, and the same vertex may occupy more than one position. We can compose two graphs when the output list of the first and the input list of the second have the same length. Their product identifies the vertices at corresponding positions in these two lists. We can also place two graphs side by side by taking their disjoint union. These two operations, called graph product and graph sum, allow directed hypergraphs to be described by expressions built from labelled hyperedges and elementary graphs. Graphoid automata use these expressions to calculate relations between words of states.

A \emph{graphoid} is the algebraic structure used for these two operations \citep{BozapalidisKalampakas2008}. Its five structural elements are $s,d_{01},d_{21},d_{10},d_{12}$. In the graph interpretation, $s$ exchanges two positions and $d_{pq}$ has one vertex occurring $p$ times as an input and $q$ times as an output. Fifteen equations govern these elements and their interaction with arbitrary graph expressions. They come from the finite graph axiomatization of \citet{BozapalidisKalampakas2004}.

To define an automaton, we assign a relation between words of states to each hyperedge label and choose structural relations satisfying the graphoid equations. We evaluate graph expressions by relational composition and parallel product. Since the chosen relations satisfy the graphoid equations, two expressions for the same graph have the same value. Initial and final regular languages specify the allowed boundary words. This construction was introduced by \citet{BozapalidisKalampakas2008}, following the recognizability results of \citet{BozapalidisKalampakas2006Recognizability}. Its algebraic operations use the magmoids of \citet{ArnoldDauchet1978}. Related constructions of graphs use graph concatenation \citep{EngelfrietVereijken1997} and hyperedge replacement \citep{DrewesHabelKreowski1997}.

Our aim is to determine which distinctions between graphs survive relational or matrix evaluation. For example, first applying $d_{01}$ and then $d_{10}$ creates a vertex and forgets its boundary occurrence. The resulting graph $d_{01}\comp d_{10}$ consists of one isolated internal vertex and differs from the empty graph $e_0$. Under the canonical relational interpretation, however, the isolated vertex can be assigned a state whenever the state set is nonempty. Both graphs therefore impose no constraint and have the same value. Matrix evaluation instead sums one unit for each possible state of the isolated vertex, and this sum may differ from the unit assigned to the empty graph.

To compare equality of graphs with equality of their values, we first show that the graphoid equations are equivalent to the special commutative Frobenius equations. The elements $d_{21},d_{01},d_{12},d_{10}$ become multiplication, unit, comultiplication and counit. The product and sum laws alone do not prescribe a symmetry, so we construct permutations from adjacent copies of $s$. Equation (G15) says that moving an extra wire past the inputs of a morphism before applying it gives the same result as moving that wire past its outputs afterwards. Applying this equation successively to a block of wires gives naturality of the symmetry.

This correspondence lets us apply known categorical presentations. \citet{RSW2005} identify the free special commutative Frobenius structure with cospans of finite sets, \citet{CoyaFong2017} give the extraspecial presentation by corelations, and \citet{BonchiGadducciKissingerSobocinskiZanasi2022} establish the corresponding description by open hypergraphs with labelled generators. We use these presentations to compare the original graphoid calculus, its bone quotient, and their interpretations.

A structural expression contains no labelled hyperedges. Its cospan is determined, up to renaming vertices, by a partition of the boundary positions and the number of isolated internal vertices. Positions belong to the same block precisely when they are attached to the same vertex. The \emph{bone law} $\dunit\comp\dcounit=e_0$ deletes the isolated vertices, leaving only the partition, or corelation.

When we put a labelled graph expression into normal form, each input and output port of a hyperedge remains attached to the same vertex. The bone law allows us to delete a vertex only if it is not on the boundary and no hyperedge is attached to it. A component that contains a labelled hyperedge therefore cannot be deleted by the bone law, even when none of its vertices is on the boundary.

For structural expressions, canonical relational evaluation distinguishes exactly the boundary partitions when there are at least two states. Matrix evaluation also depends on the isolated-vertex count. Let $Q$ be a nonempty finite state set and $K$ a commutative semiring. Each isolated vertex contributes $\lambda=|Q|\cdot1_K$, so a structural graph with partition $P$ and $c$ isolated vertices has matrix $\lambda^c\chi^K(R_P)$, where $R_P$ consists of assignments constant on each block and $\chi^K(R_P)$ is its characteristic matrix.

These matrices occur in partition categories. The maps of \citet[Definitions~1.6 and~1.7]{BanicaSpeicher2009} have entry $1$ exactly when indices agree within every block, and a closed component contributes the dimension parameter in the linear Frobenius presentation of \citet[Definition~3.1]{BrundanVargas2022}. For two states, \citet{EastJohnsonKambites2026} use the same matrix formula. Their results include faithful representations of partitions and twisted partitions, and quotients in which sufficiently many isolated vertices give zero matrices.

We use these representations to give a uniform equality criterion over an arbitrary commutative semiring and any nonempty finite state set. Two structural matrices can agree because different isolated-vertex counts give the same power of $\lambda$. When this common power is nonzero, at least two states still distinguish the boundary partitions. When it is zero, every entry vanishes and the partition can no longer be recovered. It refines the known structural representations and makes their consequences explicit for the original graphoid calculus.

The bone law requires $\lambda=1_K$. Preserving every relational composition imposes a separate condition: the product of characteristic matrices counts intermediate words, whereas relational composition asks whether any such word exists. For $|Q|\geq2$, these calculations agree for all relations exactly when addition in $K$ is idempotent. The bone law on three states can hold without this condition, so validity of the graph equations alone does not ensure preservation of arbitrary relational compositions.

Other structural relations combine states by group operations. \citet{Pavlovic2009} proved that classical structures in $\Rel$ are disjoint unions of Abelian groups, the same data used in the Abelian relational graphoids of \citet{Kalampakas2011}. We recall this correspondence through the groupoid description of \citet{HCC2013}, then determine which equations survive passage to characteristic matrices. The Frobenius equations hold over every commutative semiring. Specialness requires $|Q_i|\cdot1_K=1_K$ for each group $Q_i$, because every state in that group has $|Q_i|$ decompositions as a product. The bone law instead requires $|I|\cdot1_K=1_K$, where $I$ indexes the groups, because the unit selects one identity from each group. Singleton groups give the canonical copy and compare structure used for colourability in \citet{Kalampakas2023Colorability}. Nontrivial groups allow further state combinations, whose effect on recognition is studied in \citet{Papadopoulos2025Recognition}.

Other categorical approaches to graph recognizability use automaton functors for arrows and cospans \citep{BrugginkKoenigKupper2015,BrugginkKoenig2018}, or finite automata based on graph concatenation \citep{DrewesHoffmannMinas2025FiniteAutomata}. Here the equation correspondence allows us to use the known cospan, hypergraph and partition representations within the classical graphoid model. The resulting comparison separates equality of graphs, equality after the bone law, and equality of evaluated matrices. It also specifies when characteristic matrices preserve the compositions used by relational automata. We use the categorical account of graphical languages and hypergraph categories in \citet{Lack2004,Selinger2011,FongSpivak2019}.

After the preliminaries, Section~\ref{sec:presentation} proves the equation correspondence. Sections~\ref{sec:relational} and~\ref{sec:decorated} treat relational evaluation and labelled graphs. Sections~\ref{sec:normalforms} and~\ref{sec:completeness} give normal forms and structural completeness. Section~\ref{sec:lifting} determines matrix equality and the conditions for preserving the bone law and relational composition. Section~\ref{sec:representation} treats Abelian group structures, followed by examples and concluding remarks.

\section{Classical graphoids and categorical preliminaries}\label{sec:preliminaries}

\begin{notation}\label{not:convention}
An element of rank $(m,n)$ has $m$ input positions and $n$ output positions. We read composition from left to right, so for $f\in M_{m,n}$ and $g\in M_{n,k}$, the expression $f\comp g$ means first $f$ and then $g$. We write $\bx$ for the parallel operation, $e_n$ for the unit of rank $(n,n)$, and $e=e_1$.
\end{notation}

\begin{definition}\label{def:magmoid}
A \emph{magmoid} is a doubly ranked family of sets $M=(M_{m,n})_{m,n\geq0}$, equipped with a sequential operation and a parallel operation,
\[
\comp:M_{m,n}\times M_{n,k}\longrightarrow M_{m,k},
\qquad
\bx:M_{m,n}\times M_{p,q}\longrightarrow M_{m+p,n+q},
\]
and with a unit $e_n\in M_{n,n}$ for every $n\geq0$. The sequential operation is defined when the intermediate ranks agree, whereas the parallel operation adds the input ranks and the output ranks. Both operations are associative and satisfy the interchange law
\[
(f\comp g)\bx(h\comp k)=(f\bx h)\comp(g\bx k)
\]
whenever the composites are defined. For $f\in M_{m,n}$, the unit laws are
\[
e_m\comp f=f=f\comp e_n,\qquad
e_0\bx f=f=f\bx e_0,\qquad
e_m\bx e_n=e_{m+n}.
\]
\end{definition}

In the magmoid terminology of \citet{ArnoldDauchet1978}, the sequential and parallel operations are called product and sum. For graphs, interchange says that composing two independent pairs and then taking their disjoint union gives the same graph as taking the disjoint unions first and composing them.

A graphoid has five distinguished elements. In the graph interpretation, they exchange positions, create a vertex with an output occurrence, join input occurrences, forget an input occurrence, or repeat an output occurrence. Their equations are those of \citet{BozapalidisKalampakas2008}, based on the finite graph presentation of \citet{BozapalidisKalampakas2004}, and are also recalled in \citet{Kalampakas2023Colorability}.

\begin{definition}\label{def:graphoid}
A \emph{graphoid} is a magmoid $(M,\comp,\bx,(e_n))$ together with a distinguished set
\[
D=\{\swp,\dunit,\dmul,\dcounit,\dcomul\},
\]
where
\[
\swp\in M_{2,2},\quad \dunit\in M_{0,1},\quad
\dmul\in M_{2,1},\quad \dcounit\in M_{1,0},\quad
\dcomul\in M_{1,2},
\]
satisfying the following graph equations:

\begin{align*}
\swp\comp\swp&=e_2, \qquad\textup{(G1)}\\
(\swp\bx e)\comp(e\bx\swp)\comp(\swp\bx e)&=(e\bx\swp)\comp(\swp\bx e)\comp(e\bx\swp), \qquad\textup{(G2)}\\
(e\bx\dmul)\comp\dmul&=(\dmul\bx e)\comp\dmul, \qquad\textup{(G3)}\\
(e\bx\dunit)\comp\dmul&=e, \qquad\textup{(G4)}\\
\swp\comp\dmul&=\dmul, \qquad\textup{(G5)}\\
(e\bx\dunit)\comp\swp&=\dunit\bx e, \qquad\textup{(G6)}\\
(\swp\bx e)\comp(e\bx\swp)\comp(\dmul\bx e)&=(e\bx\dmul)\comp\swp, \qquad\textup{(G7)}\\
\dcomul\comp(e\bx\dcomul)&=\dcomul\comp(\dcomul\bx e), \qquad\textup{(G8)}\\
\dcomul\comp(e\bx\dcounit)&=e, \qquad\textup{(G9)}\\
\dcomul\comp\swp&=\dcomul, \qquad\textup{(G10)}\\
\swp\comp(e\bx\dcounit)&=\dcounit\bx e, \qquad\textup{(G11)}\\
(\dcomul\bx e)\comp(e\bx\swp)\comp(\swp\bx e)&=\swp\comp(e\bx\dcomul), \qquad\textup{(G12)}\\
\dcomul\comp\dmul&=e, \qquad\textup{(G13)}\\
(\dcomul\bx e)\comp(e\bx\dmul)&=\dmul\comp\dcomul. \qquad\textup{(G14)}
\end{align*}
Define the block permutations recursively by
\[
s_{0,1}=e,\qquad
s_{m+1,1}=(s_{m,1}\bx e)\comp(e_m\bx\swp).
\]
The element $s_{m,1}$ has rank $(1+m,m+1)$ and moves the first position past the following $m$ positions. The final axiom is
\begin{equation*}
 s_{m,1}\comp(f\bx e)=(e\bx f)\comp s_{n,1},
 \qquad f\in M_{m,n}. \qquad\textup{(G15)}
\end{equation*}
Both sides have rank $(1+m,n+1)$. On the left, the first position moves past the $m$ inputs before $f$ acts. On the right, $f$ acts first and the untouched position then moves past its $n$ outputs. Repeating this equation one position at a time gives the corresponding equations for arbitrary blocks.
A morphism of graphoids preserves ranks, the two operations, the units and each distinguished element of $D$.
\end{definition}

Let $\Sigma$ be a doubly ranked alphabet, so each letter has a prescribed number of input and output positions. We write $\GR(\Sigma)$ for the free graphoid on $\Sigma$. Its elements are expressions in these letters and the structural elements, with two expressions identified whenever the graphoid equations make them equal. By the graph presentation theorem, these elements are precisely the isomorphism classes of finite directed $\Sigma$-hypergraphs with ordered input and output lists \citep{BozapalidisKalampakas2004,BozapalidisKalampakas2008}.

Each $d_{pq}$ has one vertex, repeated at its $p$ input and $q$ output positions, and no hyperedges. In particular, $\dunit\comp\dcounit$ retains that vertex after its boundary occurrence has been forgotten, whereas $e_0$ has no vertices.

The \emph{bone law} identifies these two graphs by adding the equation
\begin{equation*}
\dunit\comp\dcounit=e_0. \qquad\textup{(B)}
\end{equation*}
We write $\GRcore(\Sigma)$ for the quotient of $\GR(\Sigma)$ by the bone law.

We now express the magmoid operations in categorical terms. Take the natural numbers as objects and take the elements of $M_{m,n}$ as morphisms from $m$ to $n$. The sequential operation becomes composition of morphisms, and the parallel operation becomes the tensor product, with addition as the tensor product of objects. The magmoid laws say exactly that these operations form a strict monoidal category.

A \PROP{} is a category of this form equipped with a natural symmetry. The magmoid laws alone do not provide maps that exchange positions. In Section~\ref{sec:presentation}, we construct them from $\swp$ using (G1), (G2), interchange and (G15). For the categorical terminology and its graphical interpretation, see \citet{Lack2004} and \citet{Selinger2011}.

The other four distinguished elements will form a multiplication, a unit, a comultiplication and a counit on the object $1$, related by the Frobenius equations.

\begin{definition}\label{def:frobenius}
Let $\mathcal C$ be a symmetric strict monoidal category. A \emph{commutative Frobenius monoid} on an object $X$ consists of a commutative monoid
\[
\mu:X\otimes X\to X,\qquad \eta:I\to X,
\]
and a cocommutative comonoid
\[
\delta:X\to X\otimes X,\qquad \varepsilon:X\to I,
\]
satisfying, in our composition convention, the Frobenius law
\[
(\delta\otimes\id_X)\comp(\id_X\otimes\mu)
=\mu\comp\delta
=(\id_X\otimes\delta)\comp(\mu\otimes\id_X).
\]
The structure is \emph{special} if
\[
\delta\comp\mu=\id_X,
\]
and \emph{extraspecial} if it is special and also satisfies
\[
\eta\comp\varepsilon=\id_I.
\]
In a dagger symmetric monoidal category it is a \emph{dagger-Frobenius monoid} if $\delta=\mu^\dagger$ and $\varepsilon=\eta^\dagger$.
\end{definition}

Under the graph interpretation, the extra equation is the bone law. For background on Frobenius algebras and classical structures, see \citet{Kock2004}, \citet{CPV2013}, \citet{Pavlovic2009}, and \citet{CoeckeKissinger2017}.

\begin{remark}\label{rem:convention-frob}
In the usual right-to-left convention, specialness is written $\mu\circ\delta=\id_X$ and the extra equation is written $\varepsilon\circ\eta=\id_I$. We use the left-to-right order of Notation~\ref{not:convention} throughout.
\end{remark}

Cospans of finite sets describe structural graphs by assigning a vertex to each boundary position. Corelations retain only which positions have the same assigned vertex.

\begin{definition}\label{def:cospan-corel}
The \PROP{} $\Cospan(\FinSet)$ has the natural numbers as objects. A morphism $m\to n$ is an isomorphism class of cospans of finite sets
\[
\{1,\ldots,m\}\longrightarrow A\longleftarrow\{1,\ldots,n\}.
\]
The apex $A$ is the vertex set, with input and output occurrences specified by the two maps. Composition is pushout over the common boundary, identifying the vertices at corresponding positions, and the monoidal product is disjoint union.

A morphism $m\to n$ in $\Corel$ is a partition of the $m+n$ boundary positions. The partition associated with a cospan puts two positions in the same block if their images in the apex agree. To compose two partitions, identify their corresponding common boundary positions. Take the equivalence relation generated by both partitions on the resulting set of positions, then restrict it to the input positions of the first and the output positions of the second. These are the external positions of the composite. The monoidal product is disjoint union of partitions.
\end{definition}

Each block of the induced partition corresponds to one apex element reached by a boundary map. Apex elements outside both boundary images contribute no blocks. Thus, to recover a cospan up to isomorphism from its corelation, we must also specify how many of these additional apex elements it has.

For evaluation, we use relations between words of states and matrices whose entries assign weights to pairs of such words.

\begin{definition}\label{def:rel-kmat}
Let $Q$ be a finite set of states, and let $Q^m$ be the set of state words of length $m$. Write $\varepsilon$ for the empty word, so $Q^0=\{\varepsilon\}$. A morphism $m\to n$ in the \PROP{} $\Rel(Q)$ is a relation $R\subseteq Q^m\times Q^n$. Composition is relational composition, so a pair of external words belongs to the composite when some intermediate word relates them through the two factors. The monoidal product is Cartesian product, with pairs of words identified with their concatenation.

For a commutative semiring $K$, the morphisms $m\to n$ in $\KMat(Q)$ are $K$-matrices indexed by $Q^m\times Q^n$. They compose by matrix multiplication, summing over the intermediate words. Their monoidal product is the Kronecker product, where words are concatenated to form the indices. The Boolean case is denoted by $\BMat(Q)$.

Relational converse gives the dagger on $\Rel(Q)$, and matrix transpose gives the dagger on $\KMat(Q)$.
\end{definition}

Finiteness of $Q$ ensures that matrix composition uses only finite sums. We allow $Q=\varnothing$ unless a statement explicitly requires one or more states.

To include letters of $\Sigma$, we equip the apex with labelled hyperedges and ordered incidence lists. As with structural graphs, no injectivity is required of the boundary maps.

\begin{definition}\label{def:open-hyp}
Let $\Sigma=(\Sigma_{m,n})_{m,n\geq0}$ be a doubly ranked alphabet. An open $\Sigma$-hypergraph from $m$ to $n$ consists of a finite vertex set $A$, boundary maps
\[
\{1,\ldots,m\}\longrightarrow A\longleftarrow\{1,\ldots,n\},
\]
and a finite set of hyperedge occurrences. Each occurrence has a label $\sigma\in\Sigma_{p,q}$, an ordered list of $p$ source vertices and an ordered list of $q$ target vertices. Either incidence list may be empty, so labels of rank $(0,q)$, $(p,0)$ or $(0,0)$ are allowed. Distinct occurrences may have the same label and the same incidence lists. Vertices may repeat in these lists, just as they may in the boundary. We compose these hypergraphs by joining corresponding boundary vertices and take their disjoint union as the monoidal product.
\end{definition}

We shall use the following presentations after establishing the equivalence of the graphoid and Frobenius equations.

\begin{theorem}\label{fact:freeprop}
The \PROP{} freely generated by a special commutative Frobenius monoid is equivalent to $\Cospan(\FinSet)$. The \PROP{} freely generated by an extraspecial commutative Frobenius monoid is equivalent to $\Corel$.
\end{theorem}

\begin{proof}
The cospan presentation is proved by \citet{RSW2005} and is also discussed in \citet{CoyaFong2017,FongSpivak2019}. The corelation presentation is proved by \citet{CoyaFong2017}.
\end{proof}

The following spider theorem describes how the structural maps simplify within each connected component.

\begin{theorem}\label{fact:spider}
In a special commutative Frobenius \PROP{}, a connected structural component with $p$ inputs and $q$ outputs is equal to the canonical spider $S_{p,q}$. Consequently, every structural diagram is a monoidal product of these spiders, up to permutations of the input and output wires. A component with no boundary gives the closed scalar spider $S_{0,0}$. In the extraspecial quotient, each closed scalar spider is equal to the scalar identity.
\end{theorem}

\begin{proof}
Associativity and coassociativity allow us to define iterated multiplication and comultiplication without choosing parentheses,
\[
\mu^{(p)}:X^{\otimes p}\to X,\qquad
\delta^{(q)}:X\to X^{\otimes q}.
\]
We take $\mu^{(0)}=\eta$, $\delta^{(0)}=\varepsilon$ and $\mu^{(1)}=\delta^{(1)}=\id_X$. Commutativity and cocommutativity make the order of the incident wires immaterial.

We now invoke the spider theorem for special commutative Frobenius monoids \citep{CoeckeKissinger2017}. For a connected component with $p$ inputs and $q$ outputs, it gives the canonical representative
\[
S_{p,q}=\mu^{(p)}\comp\delta^{(q)}:X^{\otimes p}\longrightarrow X^{\otimes q}.
\]
Applying the reduction component by component gives a tensor product of spiders. Boundary permutations group the wires of each component and restore their original order. For a component with no boundary the reduction gives
\[
S_{0,0}=\eta\comp\varepsilon:I\to I.
\]
This scalar remains in the free special theory and becomes the identity in the extraspecial quotient.
\end{proof}

\section{Graphoid axioms as Frobenius equations}\label{sec:presentation}

Table~\ref{tab:graphoid-frobenius} gives the correspondence between the graphoid elements and the Frobenius maps. We first construct the natural symmetry from $\swp$, then show that the four other elements form a special commutative Frobenius monoid.
\begin{table}[H]
\centering
\caption{The structural graphoid elements and their Frobenius interpretation}
\label{tab:graphoid-frobenius}
\begin{tabular}{@{}c c c l@{}}
\toprule
graphoid element & rank & Frobenius map & boundary operation\\
\midrule
$\swp$ & $2\to2$ & symmetry & exchange two positions\\
$\dunit$ & $0\to1$ & $\eta$ & introduce a vertex and output occurrence\\
$\dmul$ & $2\to1$ & $\mu$ & join two input occurrences\\
$\dcounit$ & $1\to0$ & $\varepsilon$ & forget an input occurrence\\
$\dcomul$ & $1\to2$ & $\delta$ & repeat an output occurrence\\
\bottomrule
\end{tabular}
\end{table}
The terms ``copy'' and ``discard'' refer to wires or to their interpreted states. In the graph itself, $\dcomul$ repeats a boundary occurrence of one vertex, and $\dcounit$ forgets an occurrence while retaining the vertex.

\begin{figure}[H]
\centering
\begin{tikzpicture}[
    x=1cm,y=1cm,
    wire/.style={line width=0.55pt},
    dot/.style={circle,fill=black,inner sep=1.35pt},
    lab/.style={font=\scriptsize},
    every node/.style={font=\small}
]

% d01: unit
\begin{scope}[shift={(0,0)}]
\node[dot] (u) at (0,0) {};
\draw[wire] (u) -- (1.0,0);
\node[lab] at (0.5,0.35) {$d_{01}=\eta$};
\node[lab] at (0.5,-0.35) {$0\to 1$};
\end{scope}

% d21: multiplication
\begin{scope}[shift={(2.4,0)}]
\node[dot] (m) at (0,0) {};
\draw[wire] (-1.0,0.35) -- (m);
\draw[wire] (-1.0,-0.35) -- (m);
\draw[wire] (m) -- (1.0,0);
\node[lab] at (0,0.72) {$d_{21}=\mu$};
\node[lab] at (0,-0.72) {$2\to 1$};
\end{scope}

% d10: counit
\begin{scope}[shift={(5.0,0)}]
\node[dot] (c) at (0,0) {};
\draw[wire] (-1.0,0) -- (c);
\node[lab] at (-0.5,0.35) {$d_{10}=\varepsilon$};
\node[lab] at (-0.5,-0.35) {$1\to 0$};
\end{scope}

% d12: comultiplication
\begin{scope}[shift={(7.2,0)}]
\node[dot] (cm) at (0,0) {};
\draw[wire] (-1.0,0) -- (cm);
\draw[wire] (cm) -- (1.0,0.35);
\draw[wire] (cm) -- (1.0,-0.35);
\node[lab] at (0,0.72) {$d_{12}=\delta$};
\node[lab] at (0,-0.72) {$1\to 2$};
\end{scope}

% s: symmetry
\begin{scope}[shift={(10.2,0)}]
\draw[wire] (-0.9,0.35) .. controls (-0.25,0.35) and (0.25,-0.35) .. (0.9,-0.35);
\draw[wire] (-0.9,-0.35) .. controls (-0.25,-0.35) and (0.25,0.35) .. (0.9,0.35);
\node[lab] at (0,0.72) {$s$};
\node[lab] at (0,-0.72) {$2\to 2$};
\end{scope}

\end{tikzpicture}
\caption{The graphoid generators, read from left to right. Dots represent vertices and wires their boundary occurrences. The counit retains its vertex after the input wire terminates.}
\label{fig:structural-generators}
\end{figure}
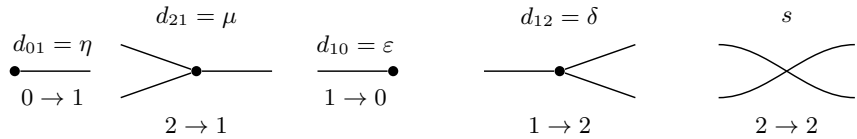

We now verify the correspondence shown in Figure~\ref{fig:structural-generators} from the classical equations.

\begin{lemma}\label{lem:dictionary}
Under the identifications
\[
\dmul=\mu,\qquad \dunit=\eta,\qquad
\dcomul=\delta,\qquad \dcounit=\varepsilon,
\]
the graphoid equations give a special commutative Frobenius monoid, with $\swp$ generating the natural symmetry. Conversely, a special commutative Frobenius structure on the generating object $1$ of a \PROP{} satisfies all fifteen graph equations. The bone equation (B) is precisely the extra equation required for extraspecialness.
\end{lemma}

\begin{proof}
For the symmetry, place $\swp$ between identity wires to exchange adjacent positions. Equations (G1) and (G2) give the involution and braid relations, and exchanges in disjoint positions commute by interchange. These are the defining relations of the symmetric group, so we obtain permutations on each tensor power. Equation (G15) allows one position to move past an arbitrary morphism $f$. Applying it to the positions of a block in turn gives naturality for all block exchanges, so the magmoid with these permutations is a \PROP{}.

Equations (G3)--(G5) give associativity, a right unit and commutativity for $\dmul$. To obtain the left unit, (G6) and involutivity give
\[
(\dunit\bx e)\comp\swp
=(e\bx\dunit)\comp\swp\comp\swp
=e\bx\dunit.
\]
Using (G5), this identity and (G4), we obtain
\[
(\dunit\bx e)\comp\dmul
=(\dunit\bx e)\comp\swp\comp\dmul
=(e\bx\dunit)\comp\dmul=e.
\]
Hence $(\dmul,\dunit)$ is a commutative monoid, with symmetry compatibility expressed by (G7).

Dually, (G8)--(G10) give coassociativity, the right counit law and cocommutativity. Equation (G11) moves the counit past the swap. Using it with cocommutativity gives the left counit law, by the dual of the unit calculation above. Thus $(\dcomul,\dcounit)$ is a cocommutative comonoid. Equation (G12) expresses its compatibility with the symmetry.

Equation (G13) gives specialness. To relate the monoid to the comonoid, substitute $\dmul=\mu$ and $\dcomul=\delta$ in (G14), obtaining
\[
(\delta\otimes\id_X)\comp(\id_X\otimes\mu)=\mu\comp\delta.
\]
Now exchange the two inputs and the two outputs on each side of this equality. Commutativity of $\mu$ and cocommutativity of $\delta$ leave $\mu\comp\delta$ unchanged. In the other composite the exchange moves $\delta$ to the second input and $\mu$ to the first output. Naturality of the symmetry therefore gives
\[
\mu\comp\delta=(\id_X\otimes\delta)\comp(\mu\otimes\id_X).
\]
This gives the Frobenius law of Definition~\ref{def:frobenius}.

Conversely, suppose that the generating object of a \PROP{} carries a special commutative Frobenius structure. Its monoid and comonoid laws give (G3), (G4), (G5), (G8), (G9) and (G10), and its specialness and Frobenius laws give (G13) and (G14). Naturality of the symmetry supplies (G6), (G7), (G11) and (G12), while the symmetric monoidal axioms give (G1), (G2) and (G15). Thus all fifteen graph equations hold. Under the same identifications, $\eta\comp\varepsilon=\id_I$ becomes $\dunit\comp\dcounit=e_0$. The extra equation is therefore exactly the bone law.
\end{proof}

The equivalence of equations in Lemma~\ref{lem:dictionary} allows us to apply the free presentations of Theorem~\ref{fact:freeprop}.

\begin{theorem}\label{thm:presentation}
The classical graphoid equations present the special commutative Frobenius \PROP{}. In magmoid notation, the equations explicitly include the generation of the symmetry and its naturality. Consequently,
\[
\GR(\varnothing)\simeq\Cospan(\FinSet).
\]
Adding the bone law gives the extraspecial commutative Frobenius \PROP{}, and hence
\[
\GRcore(\varnothing)\simeq\Corel.
\]
\end{theorem}

\begin{proof}
By Lemma~\ref{lem:dictionary}, an assignment of the structural generators satisfies the graph equations exactly when the assigned maps form a special commutative Frobenius structure. Thus the free graphoid $\GR(\varnothing)$ and the free special commutative Frobenius \PROP{} have the same universal property. The first part of Theorem~\ref{fact:freeprop} identifies the latter with $\Cospan(\FinSet)$.

We describe the interpretation explicitly to show what the vertices of a graph become in the cospan. Send each of the four one-vertex generators to a cospan with one apex point,
\[
\begin{array}{rcl}
\dunit:0\to1&\longmapsto&\varnothing\longrightarrow\{*\}\longleftarrow\{1\},\\[2mm]
\dmul:2\to1&\longmapsto&\{1,2\}\longrightarrow\{*\}\longleftarrow\{1\},\\[2mm]
\dcounit:1\to0&\longmapsto&\{1\}\longrightarrow\{*\}\longleftarrow\varnothing,\\[2mm]
\dcomul:1\to2&\longmapsto&\{1\}\longrightarrow\{*\}\longleftarrow\{1,2\}.
\end{array}
\]
The image of $\swp$ is the permutation cospan
\[
\{1,2\}\longrightarrow\{a,b\}\longleftarrow\{1,2\},
\]
whose left map sends $1$ to $a$ and $2$ to $b$, and whose right map sends $1$ to $b$ and $2$ to $a$.

Conversely, consider a finite cospan and an element of its apex. Count its input occurrences and output occurrences, say $p$ and $q$, and represent this element by the spider $S_{p,q}$. Use $S_{0,0}$ if the element has no boundary occurrence. Taking the tensor product over all apex elements and restoring the specified boundary order gives an expression for the cospan. By Theorem~\ref{fact:freeprop}, two such expressions are equal precisely when their cospans are isomorphic.

The scalar $\dunit\comp\dcounit$ represents an apex point outside the boundary images. Since each such point is a disjoint component, (B) removes them independently. The remaining cospan is determined by its boundary partition and, by the second part of Theorem~\ref{fact:freeprop}, the quotient is $\Corel$.
\end{proof}

Adding a generator for each letter of $\Sigma$ leaves the correspondence between the structural equations unchanged. Section~\ref{sec:decorated} gives the resulting hypergraph presentation.

\begin{corollary}\label{cor:syntax-quotient}
The original graph syntax $\GR(\Sigma)$ is a special Frobenius calculus of open hypergraphs. Its bone quotient $\GRcore(\Sigma)$ is the corresponding extraspecial calculus.
\end{corollary}

The graph $\dunit\comp\dcounit$ has one vertex, whereas $e_0$ has none, so they are distinct in $\GR(\Sigma)$. Thus the bone law does not follow from the original equations. Section~\ref{sec:decorated} determines its effect when labelled hyperedges are present.

\section{The relational classical structure}\label{sec:relational}

In the canonical relational interpretation, each vertex is assigned one state,
which is shared by all its incident ports. For a nonempty finite state set
$Q$, the structural operations are
\begin{align*}
\dcomul&=\{(g,gg):g\in Q\},&
\dmul&=\{(gg,g):g\in Q\},\\
\dunit&=\{(\varepsilon,g):g\in Q\},&
\dcounit&=\{(g,\varepsilon):g\in Q\},\\
\swp&=\{(g_1g_2,g_2g_1):g_1,g_2\in Q\}.
\end{align*}
The multiplication $\dmul$ compares its two input states and returns their
common value when they agree. Its converse $\dcomul$ makes two copies of a
state. Moreover, $\dunit$ allows any state to be chosen at a newly created
port, $\dcounit$ removes a port without restricting its state, and $\swp$
exchanges two ports together with their states.

These relations give the canonical $D$-set used in unitary graphoid automata.
The resulting relational graphoid is denoted $\mathrm{TSRel}(Q)$ by
\citet{Kalampakas2023Colorability}. It can also be obtained by treating each
state as a singleton group, as we explain in Section~\ref{sec:representation}.

\begin{proposition}\label{prop:rel-classical}
For every nonempty finite set $Q$, the above $D$-set forms an extraspecial
commutative dagger-Frobenius monoid in $\Rel(Q)$, with relational converse as
the dagger.
\end{proposition}

\begin{proof}
Comparison and copying are converse relations, so $\dcomul=\dmul^\dagger$. Similarly, $\dcounit=\dunit^\dagger$, and the swap is its own converse.

Either bracketing of three inputs under $\dmul$ requires all three states to agree and returns their common state. This proves associativity, and commutativity follows because the comparison is unchanged by exchanging its inputs. In the unit composite, the state supplied by $\dunit$ must equal the given input, so the composite is the identity. Taking converses gives the comonoid laws.

Each of the three composites in the Frobenius law accepts a pair of input
states and a pair of output states exactly when all four states agree.
Indeed, every intermediate state is copied from, or compared with, one of
these states. The three composites therefore give the same relation.
For specialness, $(g,h)\in\dcomul\comp\dmul$ precisely when some $u,v\in Q$ satisfy
\[
g=u=v=h.
\]
Such states exist exactly when $h=g$, so $\dcomul\comp\dmul=e$.

For the bone law there is no boundary state to prescribe the intermediate
choice. The relation $\dunit\comp\dcounit:0\to0$ first chooses a state and
then removes its port. Hence
\[
(\varepsilon,\varepsilon)\in\dunit\comp\dcounit
\quad\Longleftrightarrow\quad
\exists q\in Q.
\]
Since $Q$ is nonempty, $\dunit\comp\dcounit=\{(\varepsilon,\varepsilon)\}=e_0$, proving extraspecialness.
\end{proof}

If $Q=\varnothing$, the same structural relations still satisfy the original graphoid equations. Every relation appearing in (G1)--(G14) is then empty, and the swap gives the natural symmetry required by (G15). However, $\dunit\comp\dcounit$ is the empty relation of type $0\to0$, whereas $e_0=\{(\varepsilon,\varepsilon)\}$. Thus the structure is special but does not satisfy the bone law.

We can also represent these relations by Boolean matrices.

\begin{corollary}\label{cor:chi-models}
The characteristic-matrix map
\[
\chi\colon \Rel(Q)\longrightarrow \BMat(Q)
\]
is a canonical isomorphism of graphoid models. It sends relational composition,
monoidal product, converse, and the canonical $D$-set to their Boolean-matrix
counterparts.
\end{corollary}

\begin{proof}
For a relation $R\subseteq Q^m\times Q^n$, set
\[
\chi(R)_{x,y}=
\begin{cases}
1,& (x,y)\in R,\\
0,& (x,y)\notin R.
\end{cases}
\]
The inverse sends a Boolean matrix to the set of pairs where its entry is $1$, so $\chi$ is bijective on each set of morphisms.

Let $R\subseteq Q^m\times Q^n$ and $S\subseteq Q^n\times Q^k$. The entry of
their Boolean matrix product at $(x,z)$ is
\[
(\chi(R)\chi(S))_{x,z}
=
\bigvee_{y\in Q^n}\bigl(\chi(R)_{x,y}\wedge\chi(S)_{y,z}\bigr).
\]
This entry is $1$ exactly when some intermediate word $y\in Q^n$ satisfies $(x,y)\in R$ and $(y,z)\in S$, which is the condition for $(x,z)\in R\comp S$.

For the monoidal product, split each boundary word into the two parts belonging to its factors. The pair belongs to the product relation exactly when both resulting pairs belong to their respective relations. The Boolean Kronecker product uses the same conjunction, so $\chi$ preserves this operation. Transpose exchanges input and output words and therefore corresponds to relational converse. In particular, $\chi$ sends the five structural relations to their canonical Boolean matrices.
\end{proof}

Membership fixes every matrix entry, so the isomorphism is canonical for a fixed labelled state set. Renaming $Q$ by a permutation gives an isomorphic representation.

\section{Decorated cospans and graph automata as functors}\label{sec:decorated}

To include labels, we equip the cospan apex with the hyperedge occurrences and ordered incidence lists of Definition~\ref{def:open-hyp}. The resulting cospans are instances of decorated or structured cospans \citep{Fong2015,BaezCourser2020}. The free Frobenius presentation by labelled hypergraphs is established in \citet[Theorem~4.1 and Corollary~4.2]{BonchiGadducciKissingerSobocinskiZanasi2022}. Lemma~\ref{lem:dictionary} gives the same presentation from the classical graphoid equations.

\begin{theorem}\label{thm:decorated}
For every doubly ranked alphabet $\Sigma$, the free graphoid $\GR(\Sigma)$ is
equivalent to the \PROP{} of isomorphism classes of open $\Sigma$-hypergraphs
from Definition~\ref{def:open-hyp}. Equivalently, it is the free special
commutative Frobenius \PROP{} with one additional generator
$\sigma\colon m\to n$ for each $\sigma\in\Sigma_{m,n}$.
\end{theorem}

\begin{proof}
Write $\mathcal F(\Sigma)$ for the free special commutative Frobenius \PROP{} with the additional generators $\sigma\colon m\to n$. Lemma~\ref{lem:dictionary} turns its Frobenius generators into a graphoid structural $D$-set. The free property of $\GR(\Sigma)$ therefore gives a morphism
\[
U:\GR(\Sigma)\longrightarrow\mathcal F(\Sigma),\qquad U(\sigma)=\sigma.
\]
In the other direction, the lemma equips $\GR(\Sigma)$ with a special commutative Frobenius structure and a natural symmetry. The free property of $\mathcal F(\Sigma)$ gives a morphism
\[
V:\mathcal F(\Sigma)\longrightarrow\GR(\Sigma),\qquad V(\sigma)=\sigma.
\]
Both composites fix the structural generators and the labels. Since every expression is built from these generators, the composites are identity morphisms, so $U$ and $V$ are inverse isomorphisms.

The cited hypergraph presentation identifies $\mathcal F(\Sigma)$ with open hypergraphs, in agreement with the classical graph presentation recalled in Section~\ref{sec:preliminaries}. A vertex is represented by a spider whose legs are its boundary occurrences and incident hyperedge ports. When neither occurs, the spider is closed. Each labelled hyperedge contributes one generator occurrence, with its source and target ports in the prescribed order. Reducing the structural maps to spiders leaves every hyperedge port attached to the same vertex.

The cup $\dunit\comp\dcomul:0\to2$ and cap $\dmul\comp\dcounit:2\to0$ let us bend wires to make the required connections, including those in a directed cycle. A matching cup and cap can be straightened by a snake identity. The graph presentation ensures that two arrangements with the same incidences give the same morphism.
\end{proof}

\begin{corollary}\label{cor:decorated-core}
In the open hypergraph representation, the quotient $\GRcore(\Sigma)$ deletes
exactly the vertices with no boundary occurrence and no incident hyperedge.
Vertices incident with labelled hyperedges remain, even when their component
has no external boundary. To describe the remaining structural wiring, regard
the external positions and the hyperedge ports as its boundary. The wiring
then determines a partition of these positions and ports, and hence a corelation.
\end{corollary}

\begin{proof}
A structural component with no boundary or labelled incidence reduces to $S_{0,0}=\dunit\comp\dcounit$. The bone law replaces it by $e_0$, deleting the corresponding vertex.

We must also show that the bone law makes no further identifications. Remove from each hypergraph every vertex with neither boundary nor hyperedge incidence. A vertex removed in this way is outside both boundary images, so it could not have been joined to another vertex during composition. We may therefore remove these vertices before composing. Composition can make further vertices isolated, but removing these as well gives the same hypergraph as removing all isolated vertices after composition. Removal also commutes with disjoint union. Consequently, declaring two hypergraphs equal when their results after removal are isomorphic defines a congruence.

This congruence satisfies the bone law and leaves each hyperedge port attached to its original vertex. Conversely, every removed vertex gives a closed structural spider and can be deleted by one use of the bone law. Thus the congruence is exactly the bone quotient.
\end{proof}

To specify a graphoid automaton, we choose a relational graphoid and assign a transition relation to each label. The canonical structure of Section~\ref{sec:relational} gives unitary graphoid automata. Theorem~\ref{thm:decorated} extends these assignments to graphs.

Semantic functors from cospans of labelled ordinary graphs were developed by \citet[Section~4]{RSW2005}. The following proposition expresses the classical graphoid-automaton definition \citep{BozapalidisKalampakas2008} in functorial form, with labels having prescribed numbers of inputs and outputs.

\begin{proposition}\label{prop:automata-functors}
Equip the state set $Q$ with the canonical copy and compare relations.
A relational graphoid automaton over $\Sigma$ with this structural choice is
equivalently a \PROP{} morphism
\[
\bar\delta\colon \GR(\Sigma)\longrightarrow \Rel(Q)
\]
together with initial and final rational subsets of $Q^*$. The morphism sends
the structural $D$-set to the canonical relational Frobenius structure of
Section~\ref{sec:relational}, and each label $\sigma\in\Sigma$ to its
transition relation $\delta(\sigma)$. The rational subsets are equivalently
regular languages over the alphabet $Q$. Composing the morphism with
$\chi\colon\Rel(Q)\to\BMat(Q)$ gives the Boolean-matrix version.

For a commutative semiring $K$, a $K$-matrix graphoid automaton has the same
description with a \PROP{} morphism $\GR(\Sigma)\to\KMat(Q)$. Its structural
interpretation is given by the canonical copy/compare matrices.
\end{proposition}

\begin{proof}
Assign the canonical relations to the structural generators and set $\bar\delta(\sigma)=\delta(\sigma)$ for each label. The free property gives a unique morphism $\bar\delta$ with these values. To calculate its value on a graph, choose a term for the graph and compose the assigned relations according to that term. Any other term for the same graph is equal to it in $\GR(\Sigma)$, so the morphism gives the same relation. This is the transition extension required for the automaton.

The initial and final languages select allowed boundary words after this transition relation has been calculated, so they do not alter $\bar\delta$. Conversely, the morphism supplies the graph transitions, and the two rational languages supply the acceptance data.

Composing with the isomorphism $\chi$ from Corollary~\ref{cor:chi-models} gives the Boolean-matrix description. For general $K$, the canonical structural matrices satisfy the special Frobenius equations by Theorem~\ref{thm:special-lift}. The same free property then extends any assignment of transition matrices to a morphism into $\KMat(Q)$.
\end{proof}

We next determine when graph evaluation is unchanged by deleting an isolated internal vertex.

\begin{corollary}\label{cor:factor-core}
Every Boolean relational graphoid automaton with the canonical structure and a
nonempty finite state set factors through $\GRcore(\Sigma)$. A $K$-matrix
graphoid automaton with the canonical copy/compare structure factors through
this quotient if and only if $|Q|\cdot 1_K=1_K$.
\end{corollary}

\begin{proof}
To induce a morphism on the quotient, an interpretation must assign the same value to the two sides of its additional defining equation, $\dunit\comp\dcounit=e_0$. This condition is sufficient by the universal property of the quotient. The canonical relational structure satisfies it by Proposition~\ref{prop:rel-classical}. For matrices with nonempty $Q$, Theorem~\ref{thm:bone-scalar} calculates the left-hand scalar as $|Q|\cdot1_K$, whereas the right-hand scalar is $1_K$. If $Q$ is empty, the bone scalar is $0_K$, so the same criterion still applies.
\end{proof}

To express acceptance as a scalar, we represent the allowed boundary words by vectors and compose them with the graph transition.

\begin{corollary}\label{cor:acceptance-scalar}
Let $G\in\GR_{m,n}(\Sigma)$ and $A_G=\bar\delta(G)$. Write
$I^{(m)}\subseteq Q^m$ and $T^{(n)}\subseteq Q^n$ for the length-$m$ and
length-$n$ parts of the initial and final rational languages. Then $G$ is
accepted in the Boolean interpretation if and only if
\[
A_G\cap\bigl(I^{(m)}\times T^{(n)}\bigr)\neq\varnothing.
\]
Represent the allowed input and output words by their characteristic vectors $[I^{(m)}]\colon0\to m$ and $[T^{(n)}]\colon n\to0$. In the Boolean matrix interpretation, identify $A_G$ with its characteristic matrix. Acceptance is then equivalent to the scalar
\[
[I^{(m)}]\comp A_G\comp[T^{(n)}]
\]
having value $1$.

For weights in $K$, assign a $K$-matrix to each label and use the canonical structural matrices. Evaluating $G$ with these choices gives a matrix $A_G^K=\bar\delta_K(G)$. Put $1_K$ on the allowed boundary words and $0_K$ elsewhere. The resulting boundary vectors give the quantitative value
\[
[I^{(m)}]_K\comp A_G^K\comp[T^{(n)}]_K.
\]
This scalar sums the graph weights over all allowed boundary pairs.
\end{corollary}

Weighted boundary vectors may replace these characteristic vectors. They change how boundary words contribute to acceptance, while leaving the graph transition morphism unchanged.

\section{Normal forms}\label{sec:normalforms}

The cospan presentation specifies which boundary ports meet at the same vertex. To obtain a normal form for structural terms, we must also count the isolated internal vertices.

\begin{theorem}\label{thm:struct-nf}
A structural morphism $m\to n$ in the original graphoid $\GR(\varnothing)$ is
uniquely determined by
\begin{enumerate}[label=\textup{(\roman*)},leftmargin=2.2em]
\item a partition of its $m+n$ boundary ports into blocks, and
\item a natural number $c\geq 0$ counting the apex elements outside both boundary images, or equivalently the isolated internal vertices.
\end{enumerate}
Equivalently, the morphism is a finite cospan $m\to A\leftarrow n$, considered
up to isomorphism of its apex $A$. Thus two structural terms are equal under the original graphoid equations precisely when their boundary partitions and isolated-vertex counts agree.
\end{theorem}

\begin{proof}
By Theorem~\ref{thm:presentation}, a structural morphism is represented by a
cospan
\[
\{1,\ldots,m\}\xrightarrow{\ell} A\xleftarrow{r}\{1,\ldots,n\}.
\]
For $a\in A$, its incident boundary ports form the set
\[
B_a=
\{\,i_{\mathrm{in}}:1\leq i\leq m,\ \ell(i)=a\,\}
\cup
\{\,j_{\mathrm{out}}:1\leq j\leq n,\ r(j)=a\,\}.
\]
The nonempty sets $B_a$ partition the boundary. The remaining apex elements are counted by
\[
c=\#\{a\in A: a\text{ is not hit by either boundary map}\}.
\]
Both data are invariant under isomorphism of the apex.

Conversely, suppose the partition and the number $c$ are given. Form an apex with one element for each block and $c$ additional elements. Map each boundary port to the element for its block. This constructs a cospan with the required data. In any other cospan with these data, the apex elements reached by the boundary maps are in bijection with the same blocks. The remaining $c$ elements may be matched in any order. These bijections give an isomorphism of cospans.
\end{proof}

\begin{corollary}\label{cor:core-nf}
A structural morphism $m\to n$ in $\GRcore(\varnothing)$ is uniquely determined
by a partition of its $m+n$ boundary ports. Two structural terms are equal in
this quotient if and only if they induce the same boundary partition.
\end{corollary}

\begin{proof}
Write the cospan normal form as $(P,c)$, where $P$ is its boundary partition. Deleting an isolated internal vertex leaves $P$ unchanged, so the bone law gives $(P,c+1)=(P,c)$ for all $c\geq0$. Thus every normal form reduces to $(P,0)$. Distinct partitions remain distinct by the corelation presentation of Theorem~\ref{thm:presentation}.
\end{proof}

To compute this normal form, we use a union-find structure to keep track of the port occurrences attached to each vertex.

\begin{proposition}\label{prop:nf-complexity}
Let a structural term be given by a syntax tree of encoded size $N$, where
fixed structural generators and nodes for sequential or parallel composition
have unit cost, and an identity leaf $e_k$ has cost $k$. Its cospan normal form can be computed in
$O(N\alpha(N))$ time using union-find, where $\alpha$ is the inverse Ackermann
function. Once the normal forms have been computed, their boundary partitions
can be compared in time linear in the number of boundary ports. Comparison in
the original cospan calculus also requires comparison of the numbers of
boundary-free components.
\end{proposition}

\begin{proof}
Process the syntax tree from its leaves to the root. Each fixed generator introduces a bounded number of ports and apex points, and an identity $e_k$ introduces $k$ wires. Since the encoding charges $k$ for this identity, the leaves introduce $O(N)$ items in total. Store the exposed input and output lists so that a box product concatenates them without copying their entries.

At a composition, traverse the left output list and the right input list together, uniting each pair of matching ports. These ports become internal and leave the exposed lists, so each leaf port is traversed at most once. There are consequently $O(N)$ union-find operations, taking $O(N\alpha(N))$ time in total.

At the root, the classes reached by external ports give the boundary partition. Count the remaining classes to obtain $c$. To compare two partitions without depending on the names of their classes, traverse the boundary in its fixed order and number a class when it first occurs. Two partitions are equal exactly when the resulting sequences agree, which can be checked in linear time. For equality in the original cospan calculus, compare the two values of $c$ as well.

\end{proof}

For a labelled term, the hyperedge occurrences and ordered incidence lists must also be retained. Isomorphism of these labelled structures is a separate problem for which no complexity bound is asserted here.

\begin{corollary}\label{thm:dec-nf}
A morphism of $\GR(\Sigma)$ is represented by a finite vertex set with input
and output boundary maps, together with a finite set of labelled hyperedge
occurrences and their ordered source and target lists. Two morphisms of the
same rank are equal exactly when these incidence structures are isomorphic.

Such an isomorphism matches vertices and hyperedge occurrences bijectively.
Each boundary position must be attached to corresponding vertices in the two incidence structures, and
matched hyperedges must have the same label and corresponding vertices at
each position in their source and target lists. Distinct parallel occurrences
must be matched separately, even when their labels and incidence lists agree.

In $\GRcore(\Sigma)$, the same criterion applies after deleting all vertices
with neither a boundary occurrence nor an incident hyperedge. Vertices
attached to hyperedges remain, even in components without an external
boundary.
\end{corollary}

\begin{proof}
By Theorem~\ref{thm:decorated}, a term represents an open labelled hypergraph.
At each vertex, the Frobenius equations combine the structural maps into a
single spider. During this reduction, each boundary position and each input
or output port of a hyperedge remains attached to the same vertex. The
resulting incidence structure therefore contains exactly the data used in
the isomorphism criterion of that theorem. In particular, the original
calculus retains vertices whose spiders have no legs.
Corollary~\ref{cor:decorated-core} shows that the bone law deletes exactly
these vertices, giving the criterion for the quotient.
\end{proof}

Under the canonical copy and compare interpretation, all ports at a vertex carry the same state. Assigning one state to each vertex therefore fixes the ordered input and output words at every hyperedge. These words select its transition entry. We obtain the matrix entry of the graph by multiplying these entries and summing over assignments compatible with the boundary words.

\begin{proposition}\label{prop:state-sum}
Let \(H:m\to n\) be an open \(\Sigma\)-hypergraph with vertex set \(V\) and boundary maps
\[
\ell:\{1,\ldots,m\}\to V,
\qquad
r:\{1,\ldots,n\}\to V.
\]
Let \(Q\) be a finite state set and \(K\) a commutative semiring. Interpret the structural generators by the canonical copy and compare matrices over \(K\), and each label \(\sigma\in\Sigma_{p,q}\) by a matrix
\[
M_\sigma:Q^p\times Q^q\to K.
\]
For each hyperedge occurrence \(a\), let \(\sigma_a\in\Sigma_{p_a,q_a}\) be its label, with ordered source and target lists
\[
(u^a_1,\ldots,u^a_{p_a}),
\qquad
(v^a_1,\ldots,v^a_{q_a}).
\]
Then, for boundary words \(x=(x_1,\ldots,x_m)\in Q^m\) and \(y=(y_1,\ldots,y_n)\in Q^n\),
\[
\sem{H}_{x,y}
=
\sum_{\substack{\alpha:V\to Q\\
\alpha(\ell(i))=x_i\ (1\le i\le m)\\
\alpha(r(j))=y_j\ (1\le j\le n)}}
\prod_a
M_{\sigma_a}
\Bigl(
(\alpha(u^a_1),\ldots,\alpha(u^a_{p_a})),
(\alpha(v^a_1),\ldots,\alpha(v^a_{q_a}))
\Bigr).
\]
The product is \(1_K\) when there are no hyperedges, and the sum is \(0_K\) when no assignment satisfies the boundary conditions. When \(V=\varnothing\), there is one empty assignment.
\end{proposition}

\begin{proof}
Fix the boundary words \(x,y\) and expand an expression for \(H\). Sequential composition sums over the state words on joined wires, while parallel composition multiplies entries using the Kronecker product. Thus the required entry is a sum of products of generator entries.

We can simplify this sum by using the spider normal form. At a vertex with attached ports, the copy and compare matrices require all these ports to carry the same state. Any choice that violates this requirement contributes \(0_K\). We may therefore describe the choices satisfying these constraints by a single state \(\alpha(v)\) at each vertex \(v\). An isolated vertex also requires one state choice: although its spider has no legs, the expression \(\dunit\comp \dcounit\) sums over the intermediate state between these two maps. Thus the choices allowed by the structural maps correspond exactly to functions \(\alpha:V\to Q\).

The fixed boundary words restrict which functions we may use. Since input position \(i\) is attached to \(\ell(i)\), we need \(\alpha(\ell(i))=x_i\), and similarly \(\alpha(r(j))=y_j\) at output position \(j\). If two boundary positions attached to the same vertex prescribe different states, these requirements are inconsistent. There is then no compatible assignment, so the matrix entry is \(0_K\).

For a compatible assignment \(\alpha\), the states on the ordered source and target lists select the corresponding entry of \(M_{\sigma_a}\) at occurrence \(a\). All entries of the structural generators used by this assignment are \(1_K\). Its weight is therefore the product of the selected transition entries, with one factor for each hyperedge occurrence. Summing these weights over the compatible assignments gives the formula. Since multiplication in \(K\) is commutative, the product is independent of the order in which the hyperedge occurrences are listed.
\end{proof}

At an isolated internal vertex, neither a boundary condition nor a hyperedge transition restricts the state. Every state gives the same contribution from the rest of the graph, so summing over this state introduces the factor
\[
\sum_{q\in Q}1_K=|Q|\cdot1_K.
\]

We can also use the formula to evaluate a graph a few hyperedges at a time. Suppose that some hyperedges have already been processed. If a vertex is also attached to an unprocessed hyperedge, we still need its state to evaluate that hyperedge later. We therefore keep a table with an entry for each assignment of states to these shared vertices. Each entry is the sum of the weights from the processed hyperedges, after summing over the states at vertices that are no longer needed. The prescribed boundary values restrict the assignments throughout this calculation.

When we include another hyperedge, we extend the assignments to any new vertices and multiply by the corresponding transition entry. Once every hyperedge attached to a vertex has been processed, we can sum over its state, subject to its prescribed boundary value if it has one. Isolated internal vertices supply the factors calculated above. This procedure evaluates the state sum without forming the full Kronecker product. Because the ports at a vertex share one state, the intermediate table needs only one state index for that vertex. The order of processing therefore matters: keeping fewer vertices shared with the unprocessed hyperedges requires smaller tables.

\section{Soundness and completeness for connectivity}\label{sec:completeness}

We now ask when equality of interpreted values implies that two expressions can be proved equal from the graphoid equations. We first establish soundness for Frobenius models, then examine its converse for structural terms under the canonical relational interpretation.

\begin{theorem}\label{thm:soundness}
Every equation derivable from the graphoid axioms is valid in every special
commutative Frobenius model. After adding the bone law, every derivable equation
is valid in every extraspecial commutative Frobenius model, including the
canonical relational model $\Rel(Q)$ for every nonempty $Q$.
\end{theorem}

\begin{proof}
Let $F$ interpret terms in a special commutative Frobenius model. Every graphoid axiom holds by Lemma~\ref{lem:dictionary}. Substitution of morphisms of the required types, composition and monoidal product preserve equality. Induction on a derivation therefore gives $F(f)=F(g)$ whenever $f=g$ is derivable.

In an extraspecial model the bone law also holds, so the same induction applies to the enlarged system. The canonical relational model is extraspecial for nonempty $Q$ by Proposition~\ref{prop:rel-classical}.
\end{proof}

In the canonical relational model on a nonempty state set, an isolated internal vertex can always be assigned a state and imposes no boundary restriction. Relational evaluation therefore forgets the isolated-vertex count in Theorem~\ref{thm:struct-nf}. We compare this evaluation with the bone quotient, whose normal forms retain only the boundary partition. For two states, the resulting faithful interpretation is the Boolean case of \citet[Theorem~3.11]{EastJohnsonKambites2026}. We include the separation argument to express this fact directly as completeness for the graphoid equations.

\begin{theorem}\label{thm:completeness}
Let $|Q|\geq 2$. For structural terms $s,t\colon m\to n$, the following
conditions are equivalent:
\begin{enumerate}[label=\textup{(\alph*)},leftmargin=2.2em]
\item $s=t$ is derivable from the graphoid axioms together with the bone law.
\item $s$ and $t$ determine the same corelation on their $m+n$ boundary ports.
\item $s$ and $t$ have the same interpretation in the canonical relational
model $\Rel(Q)$.
\end{enumerate}
\end{theorem}

\begin{proof}
By the corelation part of Theorem~\ref{thm:presentation}, adding the bone law
identifies precisely the terms with the same boundary partition. This proves
the equivalence of (a) and (b). We next show that the relational interpretation
determines that partition.

For a partition $P$ of the boundary ports $B$, the interpreted relation consists of the state assignments constant on each block:
\[
R_P
=
\{\alpha:B\to Q:
 b\sim_P b'\Rightarrow \alpha(b)=\alpha(b')\}.
\]
Different blocks may receive the same state, and equal partitions clearly give equal relations.

If $P\neq P'$, then, after exchanging their roles if necessary, there are ports $i,j$ such that
\[
i\sim_P j
\qquad\text{and}\qquad
 i\not\sim_{P'}j.
\]
Choose distinct states $0,1\in Q$. Assign $0$ to the $P'$-block of $i$, $1$ to the $P'$-block of $j$, and a constant state to each remaining block. This assignment belongs to $R_{P'}$ but violates the equality required by $P$ at $i,j$. Thus $R_P\neq R_{P'}$, proving (b) equivalent to (c).
\end{proof}

With one state, every partition admits the unique boundary assignment. For example,
$\{\{1_{\mathrm{in}},1_{\mathrm{out}}\}\}$ and
$\{\{1_{\mathrm{in}}\},\{1_{\mathrm{out}}\}\}$ are distinct corelations $1\to1$, but their relations agree on a singleton set. Hence, among nonempty state sets, at least two states are necessary and sufficient to separate corelations.

The canonical structural operations impose only equalities between states. We can therefore describe exactly the part of $\Rel(Q)$ that they generate.

\begin{proposition}\label{prop:limits}
For nonempty $Q$, the sub-\PROP{} of $\Rel(Q)$ generated by the canonical
classical structure is exactly the image of $\Corel$. If $|Q|\geq2$, this is a
proper sub-\PROP{} of $\Rel(Q)$.
\end{proposition}

\begin{proof}
By the corelation normal form, every canonical structural expression has the relation $R_P$ determined by its boundary partition $P$. Conversely, a spider for each block, together with the required boundary permutations, represents $R_P$. Thus the generated sub-\PROP{} is the image of $\Corel$.

At type $1\to1$ there are only two boundary partitions. They give the equality relation $\{(q,q):q\in Q\}$ and the total relation $Q\times Q$. Choose a state $0\in Q$. When $|Q|\geq2$, the singleton relation $\{(0,0)\}$ is neither equality nor the total relation, so it lies outside the generated sub-\PROP{}.
\end{proof}

Labelled transitions may be assigned relations beyond those generated by the canonical structure. The equality criterion in Corollary~\ref{thm:dec-nf} concerns individual graphs in the free labelled calculus. Language equivalence asks a different question: whether two automata, with their chosen transitions and boundary languages, accept the same set of graphs. That question, and the related problem of minimization, requires a separate analysis.

\section{Semiring lifting}\label{sec:lifting}

Let $Q$ be finite, and let $K$ be a commutative semiring. The characteristic matrix of a relation has entry $1_K$ on its members and $0_K$ elsewhere. In a composite, matrix multiplication adds one contribution per intermediate assignment, while relational composition asks whether any such assignment exists.

\begin{theorem}\label{thm:special-lift}
Let $K$ be a commutative semiring and let $Q$ be a finite set. The
canonical copy/compare $D$-set in $\KMat(Q)$ satisfies the original graphoid
axioms. Its multiplication, unit, comultiplication and counit form a special
commutative dagger-Frobenius monoid, and its swap is the canonical symmetry.
\end{theorem}

\begin{proof}
Fix the external states on either side of an equation (G1)--(G14). The structural equalities determine at most one compatible internal assignment. Each matrix entry is therefore $1_K$ if that assignment exists and $0_K$ otherwise. Since the two sides agree relationally, they also agree as matrices, without any condition on addition in $K$.

For example, the specialness composite has entry
\[
(\dcomul\comp\dmul)_{g,h}
=
\sum_{(u,v)\in Q^2}
(\dcomul)_{g,(u,v)}(\dmul)_{(u,v),h}.
\]
Copying requires $u=v=g$ and comparison requires $u=v=h$, so a summand contributes precisely when
\[
u=v=g=h.
\]
There is one such pair if $g=h$ and none otherwise, giving the identity matrix.

Equation (G15) also has to hold for an arbitrary matrix $f$. Here the entries of $f$ may be any elements of $K$. The permutation matrices on either side merely reorder the coordinates, and the identity wire requires its two boundary states to agree. If they agree, both sides select the same entry of $f$. Otherwise both entries are zero. This proves naturality for arbitrary weights.

Transpose exchanges copy with compare and unit with counit. Thus the dagger equations also hold. The argument includes $Q=\varnothing$: every side of (G1)--(G14) is the unique matrix with its empty index set, and the permutation argument for (G15) remains valid.
\end{proof}

For the remainder of this section, assume that $Q$ is nonempty. For the bone composite, no external state fixes the choice made by the unit. All $|Q|$ states contribute, giving the following scalar condition.

\begin{theorem}\label{thm:bone-scalar}
For the canonical copy/compare structure in $\KMat(Q)$,
\[
\dunit\comp\dcounit=(|Q|\cdot 1_K)e_0.
\]
Consequently, for this fixed state set, the bone law holds if and only if
\[
|Q|\cdot 1_K=1_K.
\]
If $K$ is additively idempotent, the bone law holds for every nonempty finite
$Q$. Conversely, requiring the bone law for all nonempty finite state sets,
or for a single two-element state set, forces $K$ to be additively idempotent.
\end{theorem}

\begin{proof}
By Proposition~\ref{prop:state-sum}, each possible state of the isolated vertex contributes $1_K$, giving the scalar $|Q|\cdot1_K$. The bone law equates this with the value $1_K$ of the empty diagram.

If addition is idempotent, every nonempty finite sum of units equals $1_K$. Conversely, the bone law for two states gives
\[
1_K+1_K=1_K.
\]
Distributivity then gives, for every $a\in K$,
\[
a+a
=a\cdot1_K+a\cdot1_K
=a(1_K+1_K)
=a\cdot1_K
=a.
\]
\end{proof}

The bone law tells us when isolated internal vertices may be removed without changing the matrix. More generally, we can ask whether different numbers of isolated vertices give different matrices, and whether the matrix determines the boundary partition.

Take $Q=\{0,1\}$ and identify a boundary word with the set of positions carrying $1$. The matrix in \citet[Definition~3.1]{EastJohnsonKambites2026} tests whether the selected input and output positions together form a union of partition blocks, so it is exactly $\chi^K(R_P)$. Their twisted partition category retains the isolated-vertex count $c$, with composition adding the existing counts and any newly isolated middle components. Its representation $2^c\chi^K(R_P)$ \citep[Theorem~4.1]{EastJohnsonKambites2026} is therefore this two-state structural interpretation. The following criterion compares values of the same construction for arbitrary $Q$ and $K$, including coincident and vanishing scalar powers.

\begin{theorem}\label{thm:matrix-equality}
Let $Q$ be a nonempty finite state set, let $K$ be a commutative semiring, and put $\lambda=|Q|\cdot1_K$. A structural term $s\colon m\to n$ with normal form $(P,c)$ has canonical matrix
\[
\sem{s}=\lambda^c\chi^K(R_P),
\]
where $R_P$ consists of the boundary assignments that are constant on each block of $P$, $\chi^K(R_P)$ is its characteristic matrix over $K$, and $\lambda^0=1_K$.

Suppose that $|Q|\geq2$, and let $t\colon m\to n$ have normal form $(P',c')$. Then
\[
\sem{s}=\sem{t}
\quad\Longleftrightarrow\quad
\lambda^c=\lambda^{c'}
\quad\text{and}\quad
\bigl(\lambda^c=0_K\ \text{or}\ P=P'\bigr).
\]
If $Q$ has one element, every structural term of a fixed type has the same one-entry matrix $(1_K)$.
\end{theorem}

\begin{proof}
Fix the input and output words. A vertex reached by the boundary maps must carry the state prescribed at its boundary positions. If two of these positions prescribe different states at the same vertex, there is no compatible assignment, and the matrix entry is $0_K$. Otherwise the boundary words determine the states at all vertices reached by the boundary maps. Only the $c$ isolated internal vertices remain to be assigned states.

Each isolated vertex has $|Q|$ choices, and each choice contributes $1_K$. Its contribution after summation is therefore $\lambda$. The choices at different isolated vertices are independent, so Proposition~\ref{prop:state-sum} gives $\lambda^c$ for every boundary assignment in $R_P$. It gives $0_K$ for every other assignment, proving the matrix formula. This argument includes the empty boundary and the case $c=0$, when there is one empty assignment of the isolated vertices.

To compare two such matrices, choose a state $q\in Q$ and give every boundary position the state $q$. This assignment satisfies every boundary partition, so equality of the matrices forces $\lambda^c=\lambda^{c'}$. When the boundary is empty, its unique assignment gives the same conclusion. Write $a$ for this common scalar. If $a=0_K$, both matrices are zero, whatever their boundary partitions.

Suppose now that $a\neq0_K$. The matrix entry is then nonzero exactly at the assignments in $R_P$. Equality of the matrices therefore implies $R_P=R_{P'}$, and the two-state separation argument in Theorem~\ref{thm:completeness} gives $P=P'$. Conversely, the formula gives equal matrices whenever the stated condition holds. Finally, when $Q$ has one element, every boundary partition admits the unique boundary assignment and $\lambda=1_K$. The formula then gives the asserted one-entry matrix.
\end{proof}

To recover a structural graph from its matrix, we must recover both the boundary partition and the isolated-vertex count. Different counts must therefore give different powers of $\lambda$. We call the interpretation faithful when equality of matrices implies equality in the original structural calculus. The two-state sufficiency result is \citet[Theorem~4.1]{EastJohnsonKambites2026}. Their characteristic-zero convention is equivalent to injectivity of $\mathbb N\to K$, as the proof below explains.

\begin{corollary}\label{cor:structural-matrix-faithful}
For a nonempty finite state set $Q$, the canonical $K$-matrix interpretation is faithful on the original structural calculus if and only if $|Q|\geq2$ and the powers
\[
1_K,\lambda,\lambda^2,\ldots
\]
are pairwise distinct. For $|Q|\geq2$, this condition on the powers is equivalent to injectivity of the canonical map
\[
\mathbb N\longrightarrow K,\qquad n\longmapsto n\cdot1_K.
\]
In particular, the interpretation is faithful over $\mathbb N$ for every fixed state set with at least two elements.
\end{corollary}

\begin{proof}
If two powers agree, the corresponding graphs consisting only of isolated vertices have the same scalar matrix, although their normal forms differ. A singleton state set also fails to distinguish the two boundary partitions of type $1\to1$. Both stated conditions are therefore necessary.

Conversely, pairwise distinct powers cannot include $0_K$, because any zero power would equal the next power. Theorem~\ref{thm:matrix-equality} therefore shows that equal matrices have equal partitions and equal powers, and hence equal isolated-vertex counts. The normal forms agree, so Theorem~\ref{thm:struct-nf} gives equality in the original calculus.

It remains to express the condition on powers in terms of the semiring. If $\mathbb N\to K$ is injective, it distinguishes the natural numbers $|Q|^c$, so the powers of $\lambda$ are pairwise distinct. Conversely, suppose $a\cdot1_K=b\cdot1_K$ for natural numbers $a<b$. For every $n\geq b$, this equality gives
\[
n\cdot1_K=(n-b+a)\cdot1_K.
\]
Repeatedly reducing the index in this way shows that every natural-number multiple of $1_K$ equals one with index smaller than $b$. There are therefore only finitely many such multiples, and the powers of $\lambda$ must repeat.
\end{proof}

Faithfulness here concerns individual structural morphisms. It does not assert linear independence of their matrices, which is a separate question when linear combinations of diagrams are allowed in a partition category.

When $\lambda=1_K$, all powers agree and the matrix depends only on the boundary partition, giving an interpretation of $\Corel$. If $|Q|\geq2$ and $0_K\neq1_K$, this interpretation is faithful by Theorem~\ref{thm:matrix-equality}. Both conditions are necessary: a singleton state set admits only one boundary assignment, and the semiring with $0_K=1_K$ gives only zero matrices.

The bone law can hold without additive idempotency. In a nonzero commutative ring of characteristic two, $3\cdot1_K=1_K$, so it holds on a three-state set. Two relations on this set may nevertheless have exactly two intermediate witnesses for a given boundary pair. Their matrix product then has entry $1_K+1_K=0_K$, whereas their relational composite has characteristic entry $1_K$. Thus the bone law does not ensure preservation of every relational composition.

\begin{theorem}\label{thm:chiK-iff}
Let $Q$ be a finite set with $|Q|\geq 2$. The characteristic-matrix map
\[
\chi^K\colon \Rel(Q)\longrightarrow \KMat(Q)
\]
preserves relational composition for all relations if and only if $K$ is
additively idempotent. Equivalently, $\chi^K$ is a \PROP{} morphism extending
the relational semantics if and only if $K$ is additively idempotent.
\end{theorem}

\begin{proof}
Suppose that addition in $K$ is idempotent, and consider relations
$R\subseteq Q^m\times Q^n$ and $S\subseteq Q^n\times Q^k$. For fixed external
tuples $x\in Q^m$ and $z\in Q^k$, matrix multiplication gives
\[
(\chi^K(R)\chi^K(S))_{x,z}
=
\sum_{y\in Q^n}
\chi^K(R)_{x,y}\chi^K(S)_{y,z}.
\]
Each tuple $y$ with $(x,y)\in R$ and $(y,z)\in S$ contributes $1_K$, and every other tuple contributes $0_K$. Finiteness and idempotency reduce the sum to $1_K$ when such a witness exists, and to $0_K$ otherwise. This is the characteristic entry of $R\comp S$.

Conversely, choose distinct $a,b\in Q$ and put
\[
R=\{(\varepsilon,a),(\varepsilon,b)\}\subseteq Q^0\times Q^1,
\qquad
S=\{(a,\varepsilon),(b,\varepsilon)\}\subseteq Q^1\times Q^0.
\]
The relational composite contains the single pair $(\varepsilon,\varepsilon)$ and its characteristic entry is $1_K$, whereas the matrix product is
\[
(\chi^K(R)\chi^K(S))_{\varepsilon,\varepsilon}=1_K+1_K.
\]
Equality with the characteristic matrix of the composite forces
\[
1_K+1_K=1_K.
\]
As in the proof of Theorem~\ref{thm:bone-scalar}, multiplying this equality by an arbitrary element of $K$ gives additive idempotency. Since characteristic matrices preserve identities, monoidal products and symmetries over every $K$, preservation of composition is also equivalent to the stated \PROP{} morphism condition.
\end{proof}

\begin{remark}\label{rem:chi-not-always-embedding}
We allow the degenerate case $0_K=1_K$ in the characteristic-matrix map. All relations of a given type then have the same matrix, so the map is not injective on morphisms, although Theorem~\ref{thm:chiK-iff} still applies. When $0_K\neq1_K$, the entries recover membership in the relation, and the map is injective on each set of morphisms. This injectivity alone does not imply preservation of composition.
\end{remark}

For weighted acceptance, consider the runs compatible with the initial and final boundary languages. One may accept when at least one run has nonzero weight, or when the sum of their weights is nonzero. Over $\mathbb N$ and $\mathbb R_{\geq0}$ these criteria agree, because a finite sum is nonzero exactly when some summand is nonzero. Where nonzero elements can sum to zero, runs may cancel in the final sum. The chosen criterion can then change the recognized language even when the transitions and structural equations are fixed.

\section{Representation theory: Abelian relational graphoids}\label{sec:representation}

Other relational graphoids combine states by a group operation: states in the same group can be multiplied, and their product is the resulting state. \citet[Theorem~4.4]{Pavlovic2009} classified these classical structures, whose data also describe the Abelian relational graphoids of \citet{Kalampakas2011}. We recall the groupoid description of \citet{HCC2013} to count the intermediate choices in their characteristic matrices.

\begin{theorem}[Pavlovic~\citep{Pavlovic2009}]\label{thm:rel-classification}
Special commutative dagger-Frobenius structures on a finite set $Q$ in $\Rel$
are in bijection with partitions
\[
Q=\bigsqcup_{i\in I} Q_i
\]
in which each $Q_i$ is equipped with an Abelian group structure. Equivalently,
these structures are disjoint unions of Abelian groups, viewed as groupoids
whose set of arrows is $Q$.
\end{theorem}

\begin{proof}
By \citet{HCC2013}, a special dagger-Frobenius structure on $Q$ gives a groupoid whose arrows are the elements of $Q$. Its multiplication relation is partial composition, and its unit selects one identity arrow per object. We determine the restriction imposed by commutativity.

Let
\[
g:x\to y
\]
be an arrow. In our left-to-right convention, the unit law gives
\[
1_xg=g.
\]
Commutativity also gives the defined product
\[
g1_x=g.
\]
For $g1_x$ to be defined, the target of $g$ must be $x$, so $y=x$. Every arrow is therefore a loop. The arrows at each object form a group under composition, and these groups partition $Q$:
\[
Q=\bigsqcup_{i\in I} Q_i.
\]
The restricted multiplication is commutative, so each $Q_i$ is Abelian.

Conversely, regard each Abelian group $Q_i$ as the arrows at one object and leave products between different groups undefined. The resulting groupoid has commutative multiplication, so the correspondence gives a special commutative dagger-Frobenius structure in $\Rel$. Its partition and group operations are exactly the data of \citet{Kalampakas2011}.
\end{proof}

If each group is a singleton, two states can be multiplied exactly when they are equal, and the product is that state. Multiplication is therefore canonical comparison, and its converse is copying. This recovers $\mathrm{TSRel}(Q)$, used in unitary graphoid automata and in the colourability automata of \citet{Kalampakas2023Colorability}. A nontrivial group also permits products of distinct states. The two-element group, for example, supplies the operations used for parity properties in the Abelian automata of \citet{Papadopoulos2025Recognition}.

Comultiplication lists the pairs whose product is a given state, and multiplication sends each pair back to that state. In characteristic matrices, specialness therefore depends on the number of such pairs. The bone composite instead counts the identity elements selected by the unit, one from each group.

\begin{proposition}\label{prop:groupoid-lifting}
Let $Q=\bigsqcup_{i\in I} Q_i$ be a nonempty finite set equipped with a
disjoint union of Abelian group structures, and interpret its relations as
characteristic matrices over a commutative semiring $K$.

Over every commutative semiring $K$, these matrices form a commutative
dagger-Frobenius monoid. On a component $Q_i$, the composite
$\dcomul\comp\dmul$ has value $|Q_i|\cdot1_K$ on the diagonal and $0_K$
elsewhere. Consequently, specialness holds on that component if and only if
$|Q_i|\cdot1_K=1_K$. Moreover, the bone scalar is $|I|\cdot1_K$, because the
unit relation selects one identity in each component. The structure is
therefore extraspecial if and only if $|Q_i|\cdot1_K=1_K$ for every $i\in I$
and $|I|\cdot1_K=1_K$. In particular, additive idempotency makes all these
conditions hold.
\end{proposition}

\begin{proof}
Group multiplication is associative and commutative. In these monoid equations, fixed inputs determine at most one intermediate product, so the relational equalities also hold for characteristic matrices over arbitrary $K$. The unit equations have the same property, because each state belongs to exactly one group and that group has a unique identity. Transpose gives the comonoid laws. In the Frobenius equation, fix two input states and two output states. If all four lie in the same group, each side has exactly one intermediate choice when the product of the inputs equals the product of the outputs, and none otherwise. If the states do not all lie in one group, both sides have value $0_K$. Hence the Frobenius equation also holds over every $K$.

Specialness can involve several intermediate choices. Fix $z$ in a component $G=Q_i$. Since comultiplication is the converse of multiplication, its intermediate pairs are
\[
(x,y)\in G\times G
\quad\text{such that}\quad
xy=z.
\]
Each $x\in G$ gives exactly one pair, with $y=x^{-1}z$, and every pair multiplies back to $z$. Thus
\[
(\dcomul\comp\dmul)_{z,z}
=
\sum_{xy=z}1_K
=
|G|\cdot1_K.
\]
Off-diagonal entries are zero, so specialness on $G$ is equivalent to
\[
|G|\cdot1_K=1_K.
\]

The unit and counit admit exactly the component identities, giving
\[
\dunit\comp\dcounit
=
\sum_{i\in I}1_K
=
|I|\cdot1_K.
\]
Hence extraspecialness holds precisely when the component conditions for specialness and $|I|\cdot1_K=1_K$ all hold. If addition in $K$ is idempotent, every nonempty finite sum of units equals $1_K$. All the groups and the index set $I$ are nonempty, so these conditions hold in that case.
\end{proof}

For a singleton group, only one pair multiplies to its state, so specialness holds over every commutative semiring. A larger group fails specialness over $\mathbb N$, since each state has more than one decomposition. Thus the singleton-group structures are exactly those in this family that remain special over every commutative semiring. Their unit chooses one identity for every state, reducing the bone condition to $|Q|\cdot1_K=1_K$.

The group structure determines how states interact at vertices, whereas acceptance also depends on the labelled transitions and boundary languages. Renaming states leaves acceptance unchanged when these data are renamed consistently. Let $\theta:Q_A\to Q_B$ be a bijection between the state sets of two unitary graphoid automata, and let $\theta^*$ act letterwise on words. Suppose it carries the initial and final languages of $A$ onto those of $B$ and satisfies
\[
(u,v)\in\delta_A(\sigma)
\quad\Longleftrightarrow\quad
(\theta^*(u),\theta^*(v))\in\delta_B(\sigma)
\]
for every label and every pair of words of the required lengths. The bijection preserves the equalities imposed by copy and compare and pairs the intermediate words used in composition. Induction on a graph expression therefore pairs its accepting runs in the two automata, proving that they accept the same graphs.

Identifying distinct states may introduce new accepting runs. If one unary transition ends at $p$ and another starts at $q\neq p$, they cannot occur consecutively in a run. After identifying $p$ and $q$, the transitions may become composable and produce a new accepting run. An automaton reduction must preserve acceptance for the labelled transitions and boundary languages as well as respect the structural operations, as discussed in Section~\ref{sec:conclusion}.

\section{Examples}\label{sec:examples}

The equality criterion for structural matrices allows us to see exactly how the choice of semiring affects the distinction between graphs. We begin with this effect, then use labelled transitions to constrain the state assignments.

\begin{example}\label{ex:isolated}
Let $Q$ have $r\geq2$ states, and let $L=\dunit\comp\dcounit$ denote an isolated internal vertex. The structural graph
\[
e_1\bx L^{\bx c}:1\to1
\]
has its input and output attached to one vertex and has $c$ further isolated vertices. Over $\mathbb N$, its matrix is $r^c$ times the identity, so different values of $c$ give different matrices. Boolean evaluation instead gives the identity for every $c$, in agreement with the bone law.

A third behaviour occurs over $\mathbb Z/4\mathbb Z$ with two states. Here $\lambda=2$ and $\lambda^2=0$. Consequently, adjoining two isolated vertices makes every entry of a structural matrix zero. For instance, $e_1$ joins the input and output, whereas $\dcounit\comp\dunit$ leaves them at different vertices. They have different boundary partitions and remain distinct after adjoining two isolated vertices to each, but both resulting matrices are zero. Thus a zero scalar can give equal matrices for graphs with different boundary partitions, as allowed by Theorem~\ref{thm:matrix-equality}.
This instance also appears in \citet[Theorem~4.3]{EastJohnsonKambites2026}.

Scalar powers can also repeat without becoming zero. Over $\mathbb Z/3\mathbb Z$ with two states, $\lambda=2$ and its powers alternate between $1$ and $2$. Theorem~\ref{thm:matrix-equality} then says that two structural matrices agree exactly when their boundary partitions agree and their isolated-vertex counts have the same parity. The matrices still distinguish the partition, although they no longer determine the number of isolated vertices.
\end{example}

\begin{example}\label{ex:decorated-state-sum}
To see how a labelled edge constrains a state assignment, take two unary
labels $a,b\in\Sigma_{1,1}$. Interpret them over a commutative semiring
$K$ by matrices $M_a,M_b\colon Q\times Q\to K$. Form an open labelled
hypergraph $H\colon1\to1$ with vertices $u,w,v$, an $a$-edge from $u$ to
$w$, and a $b$-edge from $w$ to $v$. Put the input at $\ell(1)=u$ and
the output at $r(1)=v$.

Fix boundary states $x,y\in Q$. The only state still to be assigned is
the state $z$ at the internal vertex $w$. This same state is used as the
output of the $a$-edge and the input of the $b$-edge. For a chosen $z$,
the two edges contribute the product $M_a(x,z)M_b(z,y)$. Adding over the
possible internal states gives
\[
\sem{H}_{x,y}=\sum_{z\in Q} M_a(x,z)M_b(z,y),
\]
as in Proposition~\ref{prop:state-sum}.
Thus a path gives ordinary matrix multiplication. In the Boolean case it tests for an intermediate state permitting both transitions, while over $K$ it sums their weights.
\end{example}

\begin{example}\label{ex:colorability}
The colourability automaton of \citet{Kalampakas2023Colorability} assigns
colours to vertices and requires the ends of each edge to have different
colours. Its state set is $Q=\{1,\ldots,k\}$, equipped with
$\mathrm{TSRel}(Q)$, and its ordinary directed edge label
$a\in\Sigma_{1,1}$ has transition relation
\[
\delta(a)=\{(i,j):i,j\in Q,\ i\neq j\}.
\]
The initial and final languages are $I=T=\{\varepsilon\}$, so an accepted
graph must have empty input and output boundaries.

The structural relations assign one colour to each vertex, and $\delta(a)$ requires different colours at adjacent vertices. Since this relation is symmetric, edge direction does not affect the requirement. Accepting assignments are therefore exactly the proper $k$-colourings of the closed ordinary graph.
\end{example}

\begin{example}\label{ex:z2-group}
The specialness restriction in Proposition~\ref{prop:groupoid-lifting}
already appears for the group $Q=\mathbb Z_2=\{0,1\}$. Equip $Q$ with
addition modulo $2$ and use the corresponding Abelian relational graphoid.
Its multiplication is
\[
\dmul=\{((x,y),x+y):x,y\in\mathbb{Z}_2\},
\]
the unit selects $0$, and comultiplication and counit are the respective
relational converses. These relations give a special commutative
dagger-Frobenius structure in $\Rel(Q)$.

Each element has two decompositions. For example, $0=0+0=1+1$, so the specialness composite returns $0$ through two intermediate pairs. Relational composition asks only whether such a pair exists, so the composite relates each state to itself and is the identity. Boolean matrix evaluation has the same value. Over $\mathbb N$, however, each pair contributes one unit, so both diagonal entries of $\dcomul\comp\dmul$ are $2$. Thus these characteristic matrices are special over the Boolean semiring but not over $\mathbb N$.
\end{example}

\begin{example}\label{ex:hyperedge}
Equip $Q=\mathbb Z_2$ with the canonical copy and compare structure, and use addition modulo $2$ in a labelled transition. Take $t\in\Sigma_{1,1}$ and $h\in\Sigma_{2,1}$ with
\[
\delta(t)=\{(x,x+1):x\in\mathbb{Z}_2\},
\qquad
\delta(h)=\{((x,y),z):z=x+y\}.
\]
Let $G\colon1\to1$ have vertices $u,v,w$, boundary at $u,w$, a $t$-edge from $u$ to $v$, and an $h$-edge with source list $(u,v)$ and target $w$.

Fix the input state $x$ at $u$ and the output state $z$ at $w$, and let $y$ be the state at $v$. The $t$-edge reads $x$ at $u$ and returns $y$ at $v$. The $h$-edge reads these same states in its source order $(u,v)$ and returns $z$ at $w$. Both transitions must hold, so
\[
(x,z)\in\sem{G}
\quad\Longleftrightarrow\quad
\exists y\in\mathbb{Z}_2
\quad
 y=x+1
\quad\text{and}\quad
 z=x+y.
\]
Substituting $y=x+1$ gives
\[
z=x+(x+1)=1,
\]
for either choice of $x$. Consequently,
\[
\sem{G}=\{(0,1),(1,1)\}=Q\times\{1\}.
\]
Here the shared vertex $u$ retains the input state for use in $h$. Copying enforces equality at its two incidences, while the label supplies the addition constraint.
\end{example}

\section{Discussion and conclusion}\label{sec:conclusion}

The comparison of graphoid equations and their interpretations supplies criteria for choosing weights in a graph automaton. Over $\mathbb N$, at least two states give a faithful structural interpretation. A model on the bone quotient must instead assign value $1_K$ to an isolated vertex, but may still fail to preserve arbitrary relational composition. With Abelian group operations at vertices, specialness depends on the number of decompositions of each state, while the bone law depends on the number of group identities. These checks determine which graph equations the weights respect before any labelled transitions are assigned.

A further problem is to decide whether two graphoid automata accept the same graphs, and to use this information to reduce their state sets. The automaton functors of \citet{BrugginkKoenig2018} assign finite state sets to interfaces and relations to arrows. For the unitary graphoid automata in Proposition~\ref{prop:automata-functors}, the set at interface $m$ is $Q^m$. A reduction must retain this relationship: there must be one finite set $R$ such that the new set at every interface $m$ is $R^m$, with concatenation still representing the monoidal product. The new transitions and boundary languages must also accept exactly the original graphs. It remains to determine whether these requirements yield a unique minimal automaton, once minimality has been defined.

For weighted automata, the smaller automaton must give the same scalar from Corollary~\ref{cor:acceptance-scalar} for every graph. Thus the total weight over allowed boundary words and internal state assignments must be preserved. With the canonical structure on a nonempty finite state set and an additively idempotent semiring, Theorem~\ref{thm:bone-scalar} allows us to study this question on the bone quotient. Tropical and max-min weights are examples.

Proposition~\ref{prop:limits} also leaves a question about generating relations. Which additional operations, together with the canonical equality structure, generate all of $\Rel(Q)$? The presentations used in interacting Hopf algebras and graphical linear algebra \citep{BSZ2017} suggest an approach. For relations on $Q$, such a presentation would have to specify the additional generators and the equations describing their interaction with the structural maps.

\section*{Statements and Declarations}

This research received no specific grant from any funding agency in the public, commercial or not-for-profit sectors. The author declares no competing interests. Data sharing is not applicable, as no datasets were generated or analysed during this study.

\bibliographystyle{plainnat}
\bibliography{mybib}
\enlargethispage{3\baselineskip}

\end{document}